\documentclass[12pt]{amsart}
\usepackage{amssymb,latexsym, amsmath, amscd, array, graphicx}

\swapnumbers \numberwithin{equation}{section}

\theoremstyle{plain}

\newtheorem{thm}{Theorem}[section]

\newtheorem{lem}[thm]{Lemma}

\newtheorem{conjec}[thm]{Conjecture}
\newtheorem{prop}[thm]{Proposition}
\newtheorem{cor}[thm]{Corollary}

\theoremstyle{definition}
\newtheorem{defin}[thm]{Definition}

 \newcommand{\To}{\longrightarrow}
 
 \newcommand{\Wi}{\widetilde}

\DeclareMathOperator{\cat}{{\mbox{\rm cat$_{\rm LS}$}}}

\DeclareMathOperator{\as}{{\rm asdim}}

\def\scr{\mathcal}
\def\A{{\scr A}}

\def\Z{{\mathbb Z}}

\def\e{{\mathbf e}}
\def\F{{\mathbf F}}
\def\E{{\mathbf E}}
\def\S{{\mathbf S}}

\def\1{\hbox{\rm\rlap {1}\hskip.03in{\rom I}}}
\def\Bbbone{{\rm1\mathchoice{\kern-0.25em}{\kern-0.25em}
{\kern-0.2em}{\kern-0.2em}I}}

\long\def\forget#1\forgotten{} %

\newcommand\ver[1]{\marginpar{\tiny Changed in Ver \VER}}

\newcommand{\mc}{ \text {mc}}

\date{\today}
\begin{document}

\title[On Gromov's scalar curvature conjecture]{On Gromov's scalar curvature 
conjecture}

\author[D.~Bolotov]{Dmitry Bolotov}

\author[A.~Dranishnikov]{Alexander  Dranishnikov$^{1}$}

\thanks{$^{1}$Supported by NSF, grant DMS-0604494}

\address{Dmitry Bolotov, Verkin Institute of Low Temperature
Physics, Lenina Ave, 47, Kharkov, 631103, Ukraine}
\email{bolotov@univer.kharkov.ua}

\address{Alexander N. Dranishnikov, Department of Mathematics, University
of Florida, 358 Little Hall, Gainesville, FL 32611-8105, USA}
\email{dranish@math.ufl.edu}

\subjclass[2000]
{Primary 55M30; 
Secondary 53C23,  
57N65  
}

\keywords{Positive scalar curvature, macroscopic dimension,
connective K-theory, Strong Novikov Conjecture}

\begin{abstract}
We prove the Gromov conjecture on the macroscopic dimension of the
universal covering of a closed spin manifold with a positive scalar
curvature under the following assumptions on the fundamental group.
\begin{thm}
Suppose that a discrete group $\pi$ has the following properties:

1. The Strong Novikov Conjecture holds for $\pi$.

2. The natural map $per:ko_n(B\pi)\to KO_n(B\pi)$ is injective.

Then the Gromov Macroscopic Dimension Conjecture holds true for spin
$n$-manifolds $M$ with the fundamental group $\pi_1(M)=\pi$.
\end{thm}
\end{abstract}

\maketitle \tableofcontents

\section {Introduction}
In his study of manifolds with positive scalar curvature M. Gromov
observed some large scale dimensional deficiency of their universal
coverings: For an $n$-dimensional manifold $M$, its universal
covering has to be at most $(n-2)$-dimensional from the macroscopic
point of view. For example, the product of a closed $(n-2)$-manifold
$N^{n-2}$ and the standard 2-sphere $M=N^{n-2}\times S^2$ admits a
metric of positive scalar curvature (by making the 2-sphere small).
The universal covering $\Wi M=\Wi N^{n-2}\times S^2$ looks like an
$(n-2)$-dimensional space $\Wi N^{n-2}$.  Gromov predicted similar
behavior for all manifolds with positive scalar curvature. He stated 
it in \cite{G1} as the following.
\begin{conjec}[{\bf Gromov}]\label{Grom}
For every closed Riemannian $n$-manifold $(M,g)$ with a positive
scalar curvature there is the inequality $$\dim_{\mc}(\Wi M,\tilde
g)\le n-2$$ where $(\Wi M,\tilde g)$ is the universal cover of $M$
with the pull-back metric.
\end{conjec}
Here $\dim_{\mc}$ stands for the {\em macroscopic dimension} \cite
{G1}. First time this conjecture was stated in the famous "filling"
paper \cite{G2} in a different language. In \cite{GL} the conjecture
was proved for 3-manifolds.

\begin{defin}
A map $f:X\to K$ of a metric space is called {\em uniformly
cobounded} if there is $D>0$ such that $diam(f^{-1}(y))\le D$.

A metric space $X$ has the {\em macroscopic dimension} $\dim_{\mc} X \leq n$ if
there is a uniformly cobounded proper continuous map $f:X\to K^n$ to
an $n$-dimensional polyhedron.
\end{defin}
\smallskip

In \cite{G1} Gromov asked the following questions related to his
conjecture which were stated in \cite{Bol1},\cite{Bol2} in the form
of a conjecture:

\begin{conjec}[{\bf C1}] Let $(M^n,g)$ be a closed Riemannian
$n$-manifold with torsion free fundamental group, and let
$\widetilde M^n$ be the universal covering of $M^n$ with the
pull-back metric. Suppose that $\dim_{\mc} \widetilde M^n <n$. Then

(A) If $\dim_{\mc} \widetilde M^n <n$ then $\dim_{\mc} \widetilde
M^n\le n-2$.

(B) If a classifying the universal covering map $f:M^n\to B\pi$ can
be deformed to am map with $f(M^n)\subset B\pi^{(n-1)}$, then it can
be deformed to a map with $f(M^n)\subset B\pi^{(n-2)}$.
\end{conjec}

The Conjecture {\bf C1} is proven for $n = 3$ by D. Bolotov in \cite
{Bol1}. In \cite {Bol2} it was disproved for $n>3$ by a
counterexample. It turns out that Bolotov's example does not admit a
metric of positive scalar curvature \cite{Bol3} and hence it does
not affect the Gromov Conjecture~\ref{Grom}.

\

Perhaps the most famous conjecture on manifolds of positive scalar
curvature is

{\bf The Gromov-Lawson Conjecture} \cite{GL}: {\em A closed spin
manifold $M^n$ admits a metric of positive scalar curvature if and only if
$f_*([M]_{KO})=0$ in $KO_n(B\pi)$ where $f:M^n\to B\pi$ is a
classifying map for the universal covering of $M^n$.}

J. Rosenberg connected the Gromov-Lawson conjecture with the Novikov
conjecture. Namely, he proved \cite{R} that $\alpha f_*([M]_{KO})=0$
in $KO_n(C^*(\pi))$ in the presence of positive scalar curvature where $\alpha$
is the assembly map. 

\begin{conjec}[{\bf Strong Novikov Conjecture}]\label{SNC} The analytic assembly map
$$\alpha: KO_* (B\pi)\To KO_*(C^*(\pi))$$ is a monomorphism.
\end{conjec}

Then Rosenberg and Stolz proved the Gromov-Lawson conjecture for
manifolds with the fundamental group $\pi$ which satisfies the
Strong Novikov conjecture and has the natural transformation map
$$per:ko_*(B\pi)\to KO_*(B\pi)$$ injective (\cite{RS}, Theorem 4.13).

The main goal of this paper is to prove the Gromov
Conjecture~\ref{Grom} under the Rosenberg-Stolz conditions.

\section{Connective spectra and n-connected complexes}

We refer to the textbook \cite{Ru} on the subject of spectra. We
recall that for every spectrum $\E$ there is a connective cover
$\e\to \E$, i.e., the spectrum $\e$ with the morphism $\e\to\E$ that
induces the isomorphisms for $\pi_i(\e)=\pi_i(\E)$ for $i\ge 0$ and
with $\pi_i(\e)=0$ for $i<0$ . By $KO$ we denote the spectrum for
real $K$-theory, by $ko$ its connective cover, and by $per:ko\to KO$
the corresponding transformation (morphism of spectra). We will use
both notations for an $\E$-homology of a space $X$: old-fashioned
$\E_*(X)$ and modern $ H_*(X;\E)$. We recall that $KO_n(pt)=\Z$ if
$n=0$ or $n=4$ mod 8, $KO_n(pt)=\Z_2$ if $n=1$ or $n=2$ mod 8, and
$KO_n(pt)=0$ for all other values of $n$. By $\S$ we denote the
spherical spectrum. Note that for any spectrum $\E$ there is a
natural morphism $\S\to\E$ which leads to the natural transformation
of the stable homotopy to $\E$-homology $\pi^s_*(X)\to H_*(X;\E)$.

\begin{prop}\label{n+1/n-1}
Let $X$ be an $(n-1)$ connected $(n+1)$-dimensional CW complex. Then
$X$ is homotopy equivalent to the wedge of spheres of dimensions $n$
and $n+1$ together with the Moore spaces $M(\Z_m,n)$.
\end{prop}
\begin{proof}
It is a partial case of the Minimal Cell Structure Theorem (see
Proposition 4C.1 and Example 4C.2 in\cite{Hatcher}).
\end{proof}

\begin{prop}\label{pi-ko-iso}
The natural transformation $\pi_*^s(pt)\to ko_*(pt)$ induces an
isomorphism $\pi_{n}^s(K/K^{(n-2)})\to ko_{n}(K/K^{(n-2)})$ for any CW
complex $K$.
\end{prop}
\begin{proof}
Since $\pi^s$ and $ko$ are both connective, it suffices to show that
$$\pi_{n}^s(K^{(n+1)}/K^{(n-2)})\to ko_{n}(K^{(n+1)}/K^{(n-2)})$$ is an
isomorphism. Consider the diagram generated by exact sequences of
the pair $(K^{(n+1)}/K^{(n-2)},K^{(n)}/K^{(n-2)})$
$$
\begin{CD}
\oplus\Z @>>> \pi_n^s(K^{(n})/K^{(n-2)})
@>>>\pi_n^s(K^{(n+1)}/K^{(n-2)}) @>>> 0\\
@VVV @VVV @VVV @VVV\\
\oplus \Z @>>> ko_n(K^{(n)}/K^{(n-2)}) @>>>
ko_n(K^{(n+1)}/K^{(n-2)}) @>>> 0.\\
\end{CD}
$$
Since the left vertical arrow is an isomorphism and the right
vertical arrow is an isomorphism of zero groups, it suffices to show
That $\pi_{n}^s(K^{(n)}/K^{(n-2)})\to ko_{(n)}(K^{(n)}/K^{(n-2)})$ is an
isomorphism.

Note that $\pi_{n}(S^k)\to ko_{n}(S^k)$ is an isomorphism for
$k=n, n-1$. In view of Proposition~\ref{n+1/n-1} it suffices to show
that $\pi_{n}^s(M(\Z_m,n-1))\to ko_{n}(M(\Z_m,n-1))$ is an
isomorphism for any $m$ and $n$. This follows from the Five Lemma
applied to the co-fibration $S^{n-1}\to S^{n-1}\to M(\Z_m,n-1)$.
\end{proof}

\section{Inessential manifolds}

We recall the following definition which is due to Gromov.
\begin{defin}
An $n$-manifold $M$ is called {\em essential} if it does not admit a
map $f:M\to K^{n-1}$ to an $(n-1)$-dimensional complex that induces
an isomorphism of the fundamental groups. Note that always one can
take $K^{n-1}$ to be the $(n-1)$-skeleton $B\pi^{(n-1)}$ of the
classifying space $B\pi$ of the fundamental group $\pi=\pi_1(M)$.

If a manifold is not essential, it is called {\em inessential}.
\end{defin}
The following is well-known to experts.

\begin{prop}\label{equivalent inessent}
An orientable $n$-manifold $M$ is inessential if and only if
$f_*([M])\in H_n(B\pi)$ is zero for a map $f:M\to B\pi_1(M)$
classifying the universal covering of $M$.
\end{prop}
\begin{proof}
If $M$ admits a classifying map $f:M\to B\pi^{(n-1)}$, then clearly,
$f_*([M])=0$.

Let $f_*([M])=0$ for some map $f:M\to B\pi_1(M)$ that induces an
isomorphism of the fundamental groups. Let $o_n(f)\in
H^n(M;\pi_{n-1}(F))$ be the primary obstruction to deform $f$ to the
$(n-1)$-dimensional skeleton $B\pi^{(n-1)}$ and let
$o_n(1_{B\pi})\in H^n(B\pi;\pi_{n-1}(F))$ be the primary obstruction
to retraction of $B\pi$ to the $(n-1)$-skeleton. Here $F$ denotes
the homotopy fiber of the inclusion $B\pi^{(n-1)}\to B\pi$ and
$\pi_{n-1}(F)$ is considered as a $\pi$-module. Since $f_*$ induce
an isomorphism of the fundamental groups,
$f_*:H_0(M;\pi_{n-1}(F))=\pi_{n-1}(F)_{\pi}\to
H_0(B\pi;\pi_{n-1}(F))=\pi_{n-1}(F)_{\pi}$ is an isomorphism. Then
$f_*([M]\cap o_n(f))=f_*([M])\cap o_n(1_{B\pi})=0$. By the Poincare
duality  $o_n(f)=0$.
\end{proof}

\begin{prop}\label{1-obstruction}
An orientable spin $n$-manifold $M$ is inessential if
$f_*([M]_{ko})\in ko_n(B\pi)$ is zero for a map $f:M\to B\pi$
classifying the universal covering.
\end{prop}
\begin{proof}
We assume that $M$ is given a CW complex structure with one
$n$-dimensional cell and $f(M^{(n-1)})\subset B\pi^{(n-1)}$. Let
$$c^n_f:C_n(\Wi M)=\pi_n(M,M^{(n-1)})\to\pi_{n-1}(B\pi^{(n-1)})$$ be
the primary obstruction cocycle for extending $f|_{M^{(n-1)}}$ to
the $n$-cell. In view of the $\pi$-isomorphism
$\pi_{n}(B\pi,B\pi^{(n-1)})=\pi_{n-1}(B\pi^{(n-1)})$ we may assume
that $c^n_f:\pi_n(M,M^{(n-1)})\to\pi_{n}(B\pi,B\pi^{(n-1)})$ is the
induced by $f$ homomorphism of the homotopy groups. The class
$o^n_f=[c^n_f]$ of $c^n_f$ lives in the cohomology group
$H^n(M;\pi_n(B\pi,B\pi^{(n-1)})$ with coefficients in a
$\pi$-module. By the Poincare duality
$H^n(M;\pi_n(B\pi,B\pi^{(n-1)}))=H_0(M;\pi_n(B\pi,B\pi^{(n-1)}))=
H_0(\pi;\pi_n(B\pi,B\pi^{(n-1)}))$. The later group is the group of
$\pi$-invariants of $\pi_n(B\pi,B\pi^{(n-1)})$ which is equal to the
group $\pi_n(B\pi/B\pi^{(n-1)})$. Then class $[c^n_f]$ in
$\pi_n(B\pi/B\pi^{(n-1)})$ coincides with $\bar f_*(1)$ where $\bar
f:M/M^{(n-1)}=S^n\to B\pi/B\pi^{(n-1)}$ is the induced map.

Assume that the obstruction $o^n_f=[c^n_f]\ne 0$. Then we claim that
$\bar f_*$ induces a nontrivial homomorphism for
$ko$ in dimension $n$. In view of connectivity of $ko$, 
it suffices to show this for
the map $\bar f:S^n\to B\pi^{(n+1)}/B\pi^{(n-1)}$. By
Proposition~\ref{n+1/n-1} $B\pi^{(n+1)}/B\pi^{(n-1)}=(\vee S^n)\vee(\vee
M(\Z_{m_i},n))\vee(\vee S^{n+1})$. Thus, it suffices to show that a
non-nullhomotopic maps $S^n\to S^n$ and $S^n\to M(\Z_m,n)$ induce
nontrivial homomorphisms for $ko_n$. The first is obvious, the
second follows from the homotopy excision and the Five Lemma applied
to the following diagram.
$$
\begin{CD}
\pi_n(S^n) @>>> \pi_n(S^n) @>>> \pi_n(M(\Z_m,n)) @>>> 0\\
@VVV @VVV @VVV @VVV\\
ko_n(S^n) @>>> ko_n(S^n) @>>> ko_n(M(\Z_m,n)) @>>> 0.\\
\end{CD}
$$
By the definition of the fundamental class the image of $[M]_{ko}$
is a generator in
$ko_n(M,M^{(n-1)})=ko_n(M/M^{(n-1)})=ko_n(S^n)=\Z$. Then the
following commutative diagram leads to the contradiction
$$
\begin{CD}
ko_n(M) @>f_*>> ko_n(B\pi)\\
@Vq_*VV @Vp_*VV\\
ko_n(M/M^{(n-1)}) @>\bar f_*>> ko_n(B\pi/B\pi^{(n-1)}).\\
\end{CD}
$$
\end{proof}

There are many ways to detect essentiality of manifolds. One of them
deals with the Lusternik-Schnirelmann category of $X$, $\cat X$,
which is the minimal $m$ such that $X$ admits an open cover
$U_0,\dots, U_m$ contractible in $X$.
\begin{thm}\label{essential-LS}
A closed $n$-manifold is essential if and only if its
Lusternik-Schnirelmann category equals $n$.
\end{thm}
We refer to \cite{CLOT} for the proof and more facts about the
Lusternik-Schnirelmann category. Note that $\cat X$ is estimated
from below by the cup-length of $X$ possible with twisted
coefficient and its estimated from above by the dimension of $X$.
The definition of the Lusternik-Schnirelmann category can be
reformulated in terms of existence of a section of some universal
fibration (called Ganea's fibration). The characteristic class
arising from the universal Ganea fibration over the classifying
space $B\pi$ is called the Berstein-Svartz class $\beta_{\pi}\in
H^1(\pi;I(\pi))$ of $\pi$ where $I(\pi)$ the augmentation ideal of
the group ring $\Z(\pi)$ (see \cite{Ber},\cite{Sv},\cite{CLOT}.
Formally, $\beta_{\pi}$ is the image of the generator under
connecting homomorphism $H^0(\pi;\Z)\to H^1(\pi;I(\pi))$ in the long
exact sequence generated by the short exact sequence of coefficients
$$0\to I(\pi)\to\Z(\pi)\to\Z\to 0.$$
The main property of $\beta_{\pi}$ is universality: Every cohomology
class $\alpha\in H^k(\pi;L)$ is the image of $(\beta_{\pi})^k$ under
a suitable coefficients homomorphism
$I(\pi)^k=I(\pi)\otimes\dots\otimes I(\pi)\to L$. We refer to
\cite{DR} (see also \cite{Sv}) for more details.
\begin{lem}\label{inessmap}
Let $M$ be a closed inessential $n$-manifold, $n\ge 4$, supplied with a CW
complex structure and let $\pi=\pi_1(M)$. Then $M$ admits a
classifying map $f:M\to B\pi$ of the universal covering such that
$f(M)\subset B\pi^{(n-1)}$ and $f(M^{(n-1)})\subset B\pi^{(n-2)}$.
\end{lem}
\begin{proof}
Let $f:M\to B\pi$ be the
classifying map for the universal covering of $M$. We may assume that
$f|_{M^{(2)}:M^{(2)}}\to B\pi^{(2)}$ is the identity map.
First we show that $f^*(\beta^{n-1})=0$ where 
$\beta=\beta_{\pi}$ is the Berstein-Svarz class of $\pi$. Assume
that $ f^*(\beta^{n-1})\ne 0$. Then $a=[M]\cap f^*(\beta^{n-1})\ne
0$ by the Poincare Duality Theorem~\cite{Br}. There is $u\in
H^1(X;\A)$ such that $a\cap u\ne 0$ for some local system $\A$ (see
Proposition 2.3~\cite{DKR}). Then $f^*(\beta)^{n-1}\cup u\ne 0$.
Thus the twisted cup-length of $M$ is at least $n$ and hence $\cat
M=n$. It contradicts to the Theorem~\ref{essential-LS} .

Let $g:M\to B\pi^{(n-1)}$ be a cellular classifying map. Consider
the obstruction $o^{n-1}\in H^{n-1}(M;\pi_{n-2}(B\pi^{(n-2)}))$ for
extension of $g|_{M^{(n-2)}}$ to the $(n-1)$-skeleton $M^{(n-1)}$.
Here $\pi_{n-2}(B\pi^{(n-2)})$ is considered as a $\pi$-module. Note
that $o^{n-1}=g^*(o_B^{n-1})$ where $o_B^{n-1}$ is the obstruction for
retraction of $B\pi^{(n-1)}$ to $B\pi^{(n-2)}$. In view of the
universality of the Berstein-Svarz class there is a morphisms of
$\pi$-modules $I(\pi)^{n-1}\to \pi_{n-2}(B\pi^{(n-2)})$ such that
$o_B^{n-1}$ is the image of $\beta^{n-1}$ under the induced
cohomology homomorphism. The square diagram induced by $g$ and the
fact $g^*(\beta_{\pi}^{n-1})=0$ imply that $o^{n-1}=0$. Therefore
there is a map $f':M^{(n-1)}\to B\pi^{(n-2)}$ that coincides with
$g$ on $M^{(n-3)}$. Clearly, for $n\ge 5$, this map induces an
isomorphism of the fundamental groups. It is still the case for $n=4$, 
since $g$ is the identity on the 1-skeleton and
the 2-skeletons of $M$ and $B\pi$ are taken to be the same.

We show that there is an
extension $f:M\to B\pi^{(n-1)}$. It suffices to show that the
inclusion homomorphism $\pi_{n-1}(B\pi^{(n-2)})\to
\pi_{n-1}(B\pi^{(n-1)})$ is trivial. This homomorphism coincides
with the homomorphism
$$\pi_{n-1}(E\pi^{(n-2)})\to \pi_{n-1}(E\pi^{(n-1)}).$$
Since $E\pi^{(n-2)}$ is contractible in $E\pi$, by the Cellular Approximation
Theorem it is contractible in $E\pi^{(n-1)}$. This implies that 
the inclusion homomorphism is zero.
\end{proof}

\section{The Main Theorem}

\begin{lem}\label{2-obstruction}
Suppose that a classifying map $f:M\to B\pi$ of a closed spin
$n$-manifold, $n>3$, takes the $ko_*$ fundamental class to 0,
$f_*([M]_{ko})=0$. Then $f$ is homotopic to a map $g:M\to
B\pi^{(n-2)}$.
\end{lem}
\begin{proof}
In view of Proposition~\ref{1-obstruction} we may assume that
$f(M)\subset B\pi^{(n-1)}$. In view of Lemma~\ref{inessmap} we may
additionally assume that $f(M^{(n-1)})\subset B\pi^{(n-2)}$. Also we
assume that $M$ has one $n$-dimensional cell. As in the proof of
Proposition~\ref{1-obstruction} we can say that the primary
obstruction for moving $f$ into the $(n-2)$-skeleton is defined by
the cocycle $c_f:\pi_n(M,M^{(n-1)})\to\pi_n(B\pi,B\pi^{n-2})$ which
defines the cohomology class $o_f=[c_f]$ that lives in the group of
coinvariants
$\pi_n(B\pi,B\pi^{(n-2)})_{\pi}=\pi_n(B\pi/B\pi^{(n-2)})$ and is
represented by $\bar f_*(1)$ for the homomorphism $\bar
f_*:\pi_n(M/M^{(n-1)})=\Z\to\pi_n(B\pi/B\pi^{(n-2)})$ induced by the
map of quotient spaces $\bar f:M/M^{(n-1)}\to B\pi/B\pi^{(n-2)}$.

We assume that the obstruction $[c_f]$ is nonzero. Show that $\bar
f_*:ko_n(S^n)\to ko_n(B\pi/B\pi^{n-2})$ is nontrivial to obtain a
contradiction as in the proof of Proposition~\ref{1-obstruction}.
Thus, $\bar f_*(1)$ defines a nontrivial element of
$\pi_n(B\pi/B\pi^{(n-2)})$. The restriction $n>3$ implies that $\bar
f_*(1)$ survives in the stable homotopy group. In view of
Proposition~\ref{pi-ko-iso}, the element $\bar f_*(1)$ survives in
the composition
$$\pi_n(B\pi/B\pi^{(n-2)}) \to \pi_n^s(B\pi/B\pi^{(n-2)})\to
ko_n(B\pi/B\pi^{(n-2)}).$$

The commutative diagram
$$
\begin{CD}
\pi_n(S^n) @>\bar f_*>> \pi_n(B\pi/B\pi^{(n-2)})\\
@V\cong VV @V\cong VV\\
\pi_n^s(S^n) @>\bar f_*>> \pi_n^s(B\pi/B\pi^{(n-2)})\\
@V\cong VV @V\cong VV\\
ko_n(S^n) @>\bar f_*>> ko_n(B\pi/B\pi^{(n-2)})\\
\end{CD}
$$
implies that $f_*(1)\ne 0$ for $ko_n$. Contradiction.
\end{proof}

\smallskip
The Strong Novikov Conjecture is connected to the Gromov Conjecture
by means of the following theorem which is due to J. Rosenberg.

\begin{thm}[\cite{R}]\label{Rosenberg}
Suppose $M^n$ is a spin manifold with a fundamental group $\pi$. Let
$f$ be classifying map $f:M^n\to B\pi$. If $M^n$ is a positive
scalar curvature manifold then $\alpha f_*([M^n]_{KO})=0$ where
$\alpha: KO_* (B\pi)\To KO_*(C^*_r(\pi))$ is the analytic assembly
map.
\end{thm}
\begin{thm}\label{main}
Suppose that a discrete group $\pi$ has the following properties:

1. The Strong Novikov Conjecture holds for $\pi$.

2. The natural map $per:ko_n(B\pi)\to KO_n(B\pi)$ is injective.

Then the Gromov Conjecture holds for spin $n$-manifolds $M$ with the
fundamental group $\pi_1(M)=\pi$.
\end{thm}
\begin{proof}
Let $M$ be a closed spin $n$-manifold that admits a metric with
positive scalar curvature. By Theorem~\ref{Rosenberg} $\alpha\circ
per\circ f_*([M]_{ko})=0$. The conditions on $\pi$ imply that
$f_*([M]_{ko})=0$ for the classifying map $f:M\to B\pi$. Then by
Lemma~\ref{2-obstruction} $f$ is homotopic to $g:M\to B\pi^{(n-2)}$.
The induced map of the universal covering spaces $\Wi M\to
E\pi^{(n-2)}$ produces the inequality $\dim_{\mc}\Wi M\le n-2$.
\end{proof}

\begin{cor}
The Gromov Conjecture holds for spin $n$-manifolds $M$ with the
fundamental group $\pi_1(M)=\pi$ having $cd(\pi)\le n+3$ and
satisfying the Strong Novikov Conjecture.
\end{cor}
\begin{proof}
We show that $per$ is an isomorphism in dimension $n$ in this case.
Let $\F\to ko\to KO$ be the fibration of spectra induced by the
morphism $ko\to KO$. Then $\pi_k(\F)=0$ for $k\ge 0$ and
$\pi_k(\F)=\pi_k(KO)=KO_k(pt)=0$ if $k=-1,-2,-3$ mod 8. The
Atihyah-Hirzebruch $F$-homology spectral sequence for $B\pi$ implies
that $H_n(B\pi;\F)=0$  since all entries on the $n$-diagonal in the
$E^2$-term are 0. Then the coefficient exact sequence for homology
$$H_n(B\pi;\F)\to ko_n(B\pi)\to KO_n(B\pi)\to\dots $$
implies that $per:ko_n(B\pi)\to KO_n(B\pi)$ is a monomorphism.
\end{proof}
We note that this Corollary for $cd(\pi)\le n-1$ first was proven in
\cite{Bol3}.

\begin{cor}
The Gromov Conjecture holds for spin $n$-manifolds $M$ with the
fundamental group $\pi_1(M)=\pi$ having finite $B\pi$ and with
$\as\pi\le n+3$.
\end{cor}
\begin{proof}
This is a combination of the fact that the Strong Novikov conjecture
holds true for such groups $\pi$ (\cite{Ba},\cite{DFW}), the above
Corollary, and the inequality $cd(\pi)\le \as\pi$ proven in \cite{Dr}.
\end{proof}

\begin{cor}
The Gromov conjecture holds for spin $n$-manifolds $M$ with the
fundamental group $\pi_1(M)$ equal the product of free groups
$F_1\times\dots\times F_n$. In particular, it holds for free abelian
groups.
\end{cor}
\begin{proof}
The formula for homology with coefficients in a spectrum $\E$:
$$ H_i(X\times S^1;\E)\cong H_i(X;\E)\oplus H_{i-1}(X;\E)$$
implies that if $ko_*(X)\to KO_*(X)$ is monomorphism, then
$ko_*(X\times S^1)\to KO_*(X\times S^1)$ is a monomorphism. By
induction on $m$ using the Mayer-Vietoris sequence this formula 
can be generalize to the following
$$
H_i(X\times (\vee_mS^1);\E)\cong H_i(X;\E)\oplus \bigoplus_mH_{i-1}(X;\E).$$

Therefore,
$$ko_*(X\times(\bigvee_mS^1))\to KO_*(X\times(\bigvee_mS^1))$$
is a monomorphism.
\end{proof}

\end{document}